\documentclass[smallcondensed]{article}     
\usepackage{graphicx}

\usepackage{ mathrsfs } 
\usepackage{ bbold }                
\usepackage[utf8]{inputenc}
\usepackage{geometry}

\usepackage{amsmath,amsfonts, amssymb, amsthm, mathrsfs} 
\usepackage{bbm}                    
\usepackage{graphicx,caption}               

\usepackage{natbib}
\usepackage{float}                  
\usepackage{fancyhdr}               

\usepackage{url}                    
\usepackage{textcomp}

\usepackage[]{hyperref}               

\usepackage{xcolor}
\hypersetup{
	colorlinks,
	linkcolor={red!50!black},
	citecolor={blue!50!black},
	urlcolor={blue!80!black}
}

\newcommand{\R}{\mathbb{R}}
\newcommand{\N}{\mathbb{N}}

\newcommand{\C}{\mathbb{C}}

\newcommand{\M}{\mathcal{M}}

\newcommand{\ra}{\rightarrow}
\renewcommand{\L}{\mathcal{L}}

\newcommand{\1}{\mathbb{1}}
\renewcommand{\O}{\mathcal{O}}

\title{Estimation of Stopping Times for Stopped	Self-Similar Random Processes}
\author{Viktor Schulmann  \\
	TU Dortmund  \\
}

\date{\today}

\begin{document}
	\maketitle

	\newcommand{\E}{\operatorname{E}}          
	\newcommand{\Var}{\operatorname{Var}}
	\newcommand{\Cov}{\operatorname{Cov}}	
	\renewcommand{\Re}{\operatorname{Re}}   
	\renewcommand{\Im}{\operatorname{Im}}   
	\newcommand{\MSE}{\operatorname{MSE}}   
	\renewcommand{\epsilon}{\varepsilon}
	\newcommand{\supp}{\operatorname{supp}}

	
	\newtheorem{satz}{Theorem}[]
	\newtheorem{lem}[satz]{Lemma}
	\newtheorem{kor}[satz]{Corollary}
	\newtheorem{defi}[satz]{Definition}
	\newtheorem{bsp}[satz]{Example}
	\newtheorem{bmk}[satz]{Remark}
	
	\renewcommand{\labelenumi}{\roman{enumi})}  
	
	
\begin{abstract}
Let $X=(X_t)_{t\geq 0}$ be a known process and $T$ an unknown random time independent of $X$. Our goal is to derive the distribution of $T$ based on an iid sample of $X_T$. Belomestny and Schoenmakers (2015) propose a solution based the Mellin transform in case where $X$ is a Brownian motion. Applying their technique we construct a non-parametric estimator for the density of $T$ for a self-similar one-dimensional process $X$. We calculate the minimax convergence rate of our estimator in some examples with a particular focus on Bessel processes where we also show asymptotic normality.\\
\end{abstract}

{{\bf AMS Numbers:} 62G07, 62G20, 60G18, 60G40.    	\\
	

\section{Introduction}
\cite{Belomestnyasymnorm} considered the problem of recovering the distribution of an independent random time $T$ based on iid samples from a one-dimensional Brownian motion $B$ at time $T$. \cite{comte} already considered this problem for Poisson processes. Here we use the method of \cite{Belomestnyasymnorm} and derive corresponding results for self-similar processes. We particularly focus on Bessel processes. As a consequence, we extend results from \cite{Belomestnyasymnorm} to multi-dimensional Brownian motion. This is accomplished by considering the two-norm of the multi-dimensional Brownian motion, thus reducing the problem to the case of a Bessel process which is a one-dimensional process and can be treated similarly to the case of one-dimensional Brownian motion. More specifically, we give a non-parametric estimator for the density $f_T$ of $T$. We show consistency of this estimator with respect to the $L^2$ risk and derive a polynomial convergence rate for sufficiently smooth densities $f_T$. Moreover, we show that this rate is optimal in the minimax sense. The constructed estimator is also shown to be asymptotically normal.\\The paper is organized as follows: In Section \ref{secmellin} we recapitulate the \textit{Mellin transform} which is our main tool throughout this paper. Using this transform we construct our estimator in Section \ref{selfsimilarchapter} by solving a multiplicative deconvolution problem which is related to the original problem through self-similarity of the underlying process. The use of the Mellin transform in multiplicative deconvolution problems proposed by \cite{Belomestnyasymnorm} is different to the standard approach which consists in applying a log-transformation and thus reducing the problem to an additive deconvolution problem which is usually addressed by the kernel density deconvolution technique. In Section \ref{convergenceana} we give bounds on bias and variance of the estimator in the general self-similar case. In the following two sections we lay our focus on Bessel processes and give the convergence rates of our estimator for this case (Section \ref{secbesselprozess}) and show its asymptotic normality (Section \ref{asymptoticnormality}). Section \ref{secotherprocesses} is devoted to two further examples of self-similar processes where our method yields consistent estimators. Their convergence rates are provided there. In the following Section \ref{optimalitysection} we show optimality in the minimax sense of all previously obtained rates. Some numerical examples are given in Section \ref{simulationseasy}. Finally, we collect some of the longer proofs in Section \ref{proofsec}.

\section{Mellin Transform}\label{secmellin}
In this section we recapitulate some properties of the Mellin transform from \cite{butzer}. This integral transform will be our main tool in estimation procedures of the next sections. For $a<b$ define the space
 \[\mathfrak{M}_{(a,b)} := \bigcap\limits_{c\in (a,b)} \left\{f:\mathbb{R_+}\rightarrow\mathbb{C} \middle | \int_0^\infty |f(x)|x^{c-1}dx<\infty \right\}.\]
If $f$ is the density function of an $\R_+$-valued random variable, then we have at least $f\in\mathfrak{M}_1$. Moreover, if $f:\mathbb{R_+}\rightarrow\mathbb{R}$ is locally integrable on $\mathbb{R}_+$ with
\begin{center}
 $f(x)=\left\{\begin{array}{ll} \mathcal{O}(x^{-a}), & \text{for}~~ x\rightarrow 0\\
\mathcal{O}(x^{-b}), & \text{for}~~ x\rightarrow\infty \end{array}\right.  , $
\end{center}
then $f\in\mathfrak{M}_{(a,b)}$ holds.

\begin{defi}\label{defmellintrans}
For a densitiy function $f\in \mathfrak{M}_{(a,b)}$ of a random variable $X$ define
\[ \mathcal{M}[f](s) := \mathcal{M}[X](s):= \int_0^\infty f(x)x^{s-1}dx\]
as the Mellin transform of $f$ (or of $X$) in $s\in\mathbb{C}$ with $\Re(s)\in (a,b)$.
\end{defi}
If $f\in \mathfrak{M}_{(a,b)}$, then $\mathcal{M}[f](s)$ is well defined and holomorphic on the strip ${\{s\in\C|\Re(s)\in (a,b) \}}$ according to \cite{butzer}.
\begin{bsp}\label{bspmellintrans}
	\begin{enumerate}
		\item[(i)] Consider gamma densities 
		\begin{equation}\label{gammadichte}
		 f(x)=\frac{r^{\sigma}}{\Gamma(\sigma)}x^{\sigma-1}e^{-rx} 
		\end{equation}
		for $x,\sigma,r>0$. For all $s\in\C$ with $\Re(s-\sigma+1)> 0$ we have
		$$\M[f](s)=\frac{r^{1-s}}{\Gamma(\sigma)}\Gamma(s+\sigma-1).$$
			\item[(ii)] Consider for $t>0$, $d\geq 1$ the densities
		\begin{equation}\label{besseldichte}
	 f_t(x) =  \frac{2^{1-\frac{d}{2}}t^{-\frac{d}{2}}}{\Gamma\left({d}/{2}\right)} x^{d-1} e^{-\frac{x^2}{2t}} ,\quad x> 0.
	 \end{equation}
For $\Re ({s})> 1-d$ elementary calculus shows that
			\begin{equation}\label{besselmellin}
			\mathcal{M}[f_1](s) = \frac{1}{\Gamma\left({d}/{2}\right)}\Gamma\left(\frac{s+d-1}{2}\right) 2^{\frac{s-1}{2}}. 
			 \end{equation}	
	\end{enumerate}
\end{bsp}

Similar to the well-known relation of the classical Fourier transform to sums of independent random variables, the Mellin transform behaves multiplicatively with respect to products of independent random variables:

\begin{satz}\label{mellinprodukt}
	Let $X$ and $Y$ be independent $\R_+$-valued random variables with densities $f_X\in  \mathfrak{M}_{(a,b)}$ and $f_Y\in  \mathfrak{M}_{(c,d)}$, and Mellin transforms  $\mathcal{M}[X]$ and $\mathcal{M}[Y]$ for $a<b,~c<d,~(a,b)\cap(c,d)\neq \emptyset$.	Then ${XY}$ hat a density $f_{XY}\in  \mathfrak{M}_{(a,b)\cap (c,d)}$, and 
\begin{align*}
\mathcal{M}[XY](s) =\mathcal{M}[X](s)\mathcal{M}[Y](s)
\end{align*} 
	for all ${s \in \C}$ with $\Re(s)\in(a,b)\cap (c,d)$.
\end{satz}
In the setting of Theorem \ref{mellinprodukt} it is easy to see that $f_{XY}$ is identical to 
\begin{equation}\label{mellinfaltungfkt}
(f_X \odot f_Y)(s) :=\int_0^\infty f_X\left(\frac{s}{x}\right)f_Y(x) \frac{1}{x} dx
\end{equation} 
for all ${s \in \C}$ with $\Re(s)\in(a,b)\cap (c,d)$. The function $f_X \odot f_Y$ is called \textit{Mellin convolution} of $f_X$ and $f_Y$.\\
For $a<b$ denote the space of holomorphic functions on $\{s\in\C|\Re(s)\in (a,b) \}$ by $\mathcal H (a,b)$. The mapping $\M:\mathfrak{M}_{(a,b)}\ra \mathcal H (a,b)$, $f\mapsto \M[f]$ is injective. Given the Mellin transform of a function $f$ we can reconstruct $f$:

\begin{satz} \label{invmellintrans}
For $a<\gamma<b$ let  $f\in \mathfrak{M}_{(a,b)}$.  If $$\int_{-\infty}^\infty |\mathcal{M}[f](\gamma+iv)| dv<\infty,$$
then the inversion formula
\[f(x)= \frac{1}{2\pi} \int_{-\infty}^{\infty} \mathcal{M}[f](\gamma+iv)x^{-\gamma-iv} dv\]
holds almost everywhere for $x\in\R_+$.
\end{satz}

Another important result in the theory of Mellin transforms is the \textit{Parseval formula for Mellin tranforms} (see \cite[page 108]{bleistein} for the proof):
\begin{satz}\label{parseval}
Let $f,g:\R_+\ra\R$ be measurable functions such that
$$\int_0^\infty f(x)g(x)dx$$
exists. Suppose that $\M[f](1-\cdot)$ and $\M[g](\cdot)$ are holomorphic on some vertical strip $\mathcal{S}:=\{z\in\C|a<\Re(z)<b\}$ for $a,b\in\R$. If there is a $\gamma\in(a,b)$ with  $$\int_{-\infty}^\infty |\M[f](1-\gamma-is)|ds<\infty \quad \text{and} \quad \int_{0}^\infty x^{\gamma-1} |g(x)|dx<\infty,$$
then
$$ \int_0^\infty f(x)g(x)dx=\frac{1}{2\pi i}\int_{\gamma-i\infty}^{\gamma+i\infty} \M[f](1-s)\M[g](s)ds. $$
\end{satz}

\section{Construction of the Estimator} \label{selfsimilarchapter}
We consider a real-valued stochastic process $(Y_t)_{t\geq 0}$ with càdlàg paths which is self-similar with scaling parameter $H$ (for short, $H$-ss), that is
\begin{equation}\label{defselfsimilarity}
(Y_{at})_{t\geq 0}\stackrel{d}{=}(a^H Y_t)_{t\geq 0} \quad \text{for all}~ a>0.
\end{equation}
Here, $\stackrel{d}{=}$ denotes identity of all finite dimensional distributions. Let $T\geq 0$ be a stopping time with density $f_T$ independent of $Y$. It is easy to see that the density of the random variable $Y_T$ is given by 
	\begin{equation}
p_T(y)=  \int_0^\infty g_t(y) f_T(z) dt,\quad y\in\R,\label{dichtevonyT}
	\end{equation}
where $g_t$ are the densities of $Y_t$ ($t\geq 0$).
In order to construct a non-parametric estimator for $f_T$ based on iid samples $X_1,\dots,X_n$ of $Y_T$ we use the simple consequence of \eqref{defselfsimilarity} that
\begin{equation} \label{skalmitsz}
{T}^H Y_1 \stackrel{d}{=}Y_{T}.
\end{equation}
We take the absolute value on both sides and assume $f_T\in\mathfrak{M}_{(a,b)}$ with ${0\leq a<b}$ and $|Y_1|\in\mathfrak{M}_{(0,\infty)}$, so we can apply the Mellin transform on both sides of \eqref{skalmitsz} and obtain
\begin{align*}
 \mathcal{M}[|Y_T|](s) =  \mathcal{M}[{T}^H](s) \mathcal{M}[|Y_1|](s) = \mathcal{M}[{T}]\left(Hs-H+1\right)   \mathcal{M}[|Y_1|](s)   
 \end{align*}
for ${\max\{0,\frac{a+H-1}{H}\} < \Re(s) < \frac{b+H-1}{H}}$.
Setting $z:=Hs-H+1$ we conclude that
\begin{equation} \label{skalmitmellin2}
  \mathcal{M}[T](z) =  \frac{\mathcal{M}[|Y_T|](\frac{z+H-1}{H})}{\mathcal{M}[|Y_1|](\frac{z+H-1}{H}) }, \quad {\max\{1-H,a\} < \Re(z) < b}.
\end{equation}
If the Mellin inversion formula (Lemma \ref{invmellintrans}) is applicable to $T$, we may write 
\begin{equation}\label{skalmitmellin3}
  f_{T}(x) = \frac{1}{2\pi i} \int_{\gamma-i\infty}^{\gamma+i\infty} \mathcal{M}[T](z)x^{-z} dz 
  = \frac{1}{2\pi} \int_{-\infty}^{\infty} \mathcal{M}[T](\gamma+iv)x^{-\gamma-iv} dv 
\end{equation}
for $a< \gamma < b$. Combining \eqref{skalmitmellin2} and \eqref{skalmitmellin3} we obtain the representation
\begin{equation}\label{exactsolforbessel} 
  f_{T}(x) =   \frac{1}{2\pi} \int_{-\infty}^{ \infty} \frac{\mathcal{M}[|Y_T|](\frac{\gamma+H-1+iv}{H})}{\mathcal{M}[|Y_1|](\frac{\gamma+H-1+iv}{H}) }  x^{-\gamma-iv}dv 
\end{equation}
for $\max\{1-H,a\}< \gamma < b$. In order to obtain an estimator of $f_T$ based on \eqref{exactsolforbessel} we would like to replace $\mathcal{M}[|Y_T|]$ by its empirical counterpart $$\mathcal{M}_n [|Y_{T}|](s):= \frac{1}{n}\sum\limits_{k=1}^{n} |X_k|^{s-1}.$$ However, this substitution may prevent the integral in \eqref{exactsolforbessel} from converging. Thus, we introduce a sequence $(g_n)_{n\in\N}$ with $g_n\rightarrow \infty$ (chosen later) in order to regularize our estimator. In view of \eqref{exactsolforbessel} define
\begin{equation} \label{schaetzer}
\hat{f_n}(x):=  \frac{1}{2\pi} \int_{-g_n}^{g_n} \frac{\mathcal{M}_n[|Y_T|](\frac{\gamma+H-1+iv}{H})}{\mathcal{M}[|Y_1|](\frac{\gamma+H-1+iv}{H}) }  x^{-\gamma-iv} ~dv
\end{equation}
for $x>0$ and $\max\{1-H,a\}< \gamma < b$ as an estimator for $f_T$.

\section{Convergence Analysis}\label{convergenceana}
For the sake of brevity we introduce the notation $f(x)\lesssim g(x)$ for $x\ra a$, if $f=\mathcal{O}(g)$ in the Landau notation. We write $f(x)\sim g(x)$ for $x\ra a$,
if $f(x)\lesssim g(x)$ and $g(x)\lesssim f(x)$ for $x\ra a$. For $0\leq a<b$ and $\beta\in(0,\pi)$ consider the class of densities
\begin{align*}
\mathcal{C}(\beta,a,b):= \left\{ f\in\mathfrak{M}_{(a,b)}\middle| {\hspace*{-0.6cm} \exists \tilde f:S_\beta\ra \C \text{~holomorphic~with~} \tilde f|_{\R_+}=f,
	\atop
	\tilde f(z)\lesssim z^{-a} \text{~as~}z\ra 0, f(z)\lesssim z^{-b} \text{~as~}|z|\ra \infty} \right\}, 
\end{align*} 
where
$$ S_\beta := \{z\in\C: |\arg(z)|\leq \beta \}  \quad \text{and} \quad \mathcal{C}(\beta,a,\infty):=\bigcap\limits_{b\geq 0} \mathcal{C}(\beta,a,b),$$
For the bias of the estimator \eqref{schaetzer} we have:
\begin{satz} \label{abschaetzung erwartungswert fuer klasse}
Let $(Y_t)_{t\geq 0}$ be $H$-ss with càdlàg paths. Let $T\geq 0$ be a stopping time independent of $Y$ with density $f_T$. If $f_T\in\mathcal{C}(\beta,a,b)$ with $\beta\in(0,\pi)$ and $0\leq a<b$, then
\begin{equation}\label{biasabschforc}
  \E[f_T(x)- \hat{f_n}(x)] \lesssim x^{-\gamma} e^{-\beta g_n}
\end{equation}
for all $x>0$, $\gamma\in(\max\{a,1-H\},b)$.
\end{satz}
\begin{proof}
Let $x> 0$. By Fubini's theorem and \eqref{skalmitmellin2},
\begin{align}
\E[\hat{f_n}(x)] & = \E \left[\frac{1}{2\pi} \int_{-g_n}^{g_n}\frac{\mathcal{M}_n[|Y_T|](\frac{\gamma+H-1+iv}{H})}{\mathcal{M}[|Y_1|](\frac{\gamma+H-1+iv}{H}) } x^{-\gamma-iv}~dv  \right] \nonumber\\
&= \frac{1}{2\pi} \int_{-g_n}^{g_n} \frac{\E\left[\mathcal{M}_n[|Y_T|](\frac{\gamma+H-1+iv}{H})\right]}{\mathcal{M}[|Y_1|](\frac{\gamma+H-1+iv}{H})}  x^{-\gamma-iv} ~dv .\nonumber\\
& =  \frac{1}{2\pi} \int_{-g_n}^{g_n} \frac{\mathcal{M}[|Y_T|](\frac{\gamma+H-1+iv}{H})}{\mathcal{M}[|Y_1|](\frac{\gamma+H-1+iv}{H}) }  x^{-\gamma-iv} ~dv \nonumber\\
&= \frac{1}{2\pi} \int_{-g_n}^{g_n}  \mathcal{M}[T](\gamma+iv) x^{-\gamma-iv}  dv.\label{erwwertvonschaetzer}
\end{align}
We combine Theorem \ref{invmellintrans} with \eqref{erwwertvonschaetzer} to get
\begin{equation}\label{abschbias}
|f_T(x)-\E[\hat{f_n}(x)] |\leq \frac{x^{-\gamma}}{2\pi} \int_{\{|v|>g_n\} } \left|\mathcal{M}[T](\gamma+iv) \right| dv.
\end{equation}
Since $f_T\in\mathcal{C}(\beta,a,b)$ implies $\M[T](\gamma+iv) \lesssim e^{-\beta|v|}$ for $v\ra \pm \infty$, $\gamma\in(a,b)$ (see Proposition 5 in \cite{flajolet}), we have 
\begin{equation}\label{Lkonstante}
L:=  \int_{\{|v|>g_n\} } e^{\beta|v|}\left|\mathcal{M}[T](\gamma+iv) \right| dv<\infty.
\end{equation} 
Moreover, \eqref{abschbias} gives
\begin{align*}
f_T(x)- \E [\hat{f_n}(x)]
&\leq e^{-\beta g_n} \frac{x^{-\gamma}}{2\pi}  \int_{\{|v|>g_n\} } e^{\beta|v|}\left|\mathcal{M}[T](\gamma+iv) \right| dv \leq e^{-\beta g_n} x^{-\gamma}\frac{L}{2\pi}  
\end{align*}
for all $x>0$, which is our claim.
\end{proof}
Having established an upper bound on the bias of $\hat{f}_n$, we now shall do the same for the variance of our estimator. 
\begin{satz}\label{abschaetzung varianz}
Let $(Y_t)_{t\geq 0}$ be $H$-ss with càdlàg paths. Let $T\geq 0$ be a stopping time independent of $Y$ with density $f_T$. If $f_T\in\mathfrak{M}_{(a,b)}$ with ${0\leq a< b}$ and $|Y_1|\in\mathfrak{M}_{(0,\infty)}$, then
 \begin{align*}
 \Var[ \hat f_n(x) ] \leq  \frac{\M[|Y_1|](\frac{2\gamma-2}{H}+1) \M[T](2\gamma-1)}{4\pi^2 x^{-2\gamma} n}
  \left(\int_{-g_n}^{g_n}  \frac{1}{|\mathcal{M}[|Y_1|](\frac{\gamma+H-1+iv}{H})|}   dv \right)^2
 \end{align*}
  for all $n\in\N$ and all $x>0$.
\end{satz}
\begin{proof}
Let $x> 0$, $n\in\N$. As 
$$\Var \left[\int_a^b f_v dv\right] \leq \left(\int_a^b \sqrt{\Var [f_v]}dv\right)^2,$$
for any bounded random function $f_v$ (continuous in $v$), we obtain
\begin{align}
\Var[ \hat{f_n}(x)]  &\leq \frac{1}{4\pi^2 x^{-2\gamma}} \left(\int_{-g_n}^{g_n}  \frac{\sqrt{\Var[\mathcal{M}_n[|Y_T|](\frac{\gamma+H-1+iv}{H})]}}{|\mathcal{M}[|Y_1|](\frac{\gamma+H-1+iv}{H}) |}   dv \right)^2\nonumber\\
&=  \frac{1}{4\pi^2 x^{-2\gamma} n} \left(\int_{-g_n}^{g_n}  \frac{\sqrt{\Var[|Y_T|^\frac{\gamma-1+iv}{H}]} }{|\mathcal{M}[|Y_1|](\frac{\gamma+H-1+iv}{H})|}   dv \right)^2.\label{unglinproofvonvarabsch}
\end{align}
In order to get a bound on $\Var[|Y_T|^\frac{\gamma-1+iv}{H}]$ we use the self-similarity of $Y$ to get
\begin{align*}
\Var[|Y_T|^\frac{\gamma-1+iv}{H}] &\leq \E [(|Y_T|)^{(2\gamma-2)/H}]\\
&= \int_0^\infty \E [(|Y_t|)^{(2\gamma-2)/H}] f_T(t) dt \\
&= \E [(|Y_1|)^{(2\gamma-2)/H}] \int_0^\infty  t^{2\gamma-2} f_T(t) dt,
\end{align*}
which (together with \eqref{unglinproofvonvarabsch}) gives the desired bound on $\Var[\hat{f_n}(x)]$.
\end{proof}

\section{Application to Bessel Processes} \label{secbesselprozess}
In this section we choose $Y$ to be a Bessel process $BES=(BES_t)_{t\geq 0}$ starting in $0$ with dimension $d\in [1,\infty)$. Note that the case $d=1$ leads to the absolute value of the one-dimensional Brownian motion and was already considered in \cite{Belomestnyasymnorm}. We refer to \cite{yor} for detailed information about Bessel processes. It is well-known, that Bessel processes are $\frac{1}{2}$-ss and have continuous paths. Marginal densities are given by:
\[ f_t(y)= \frac{2^{1-\frac{d}{2}}t^{-\frac{d}{2}}}{\Gamma\left({d}/{2}\right)}  y^{d-1} e^{-\frac{y^2}{2t}} ,\quad y\geq 0.\]
In Example \ref{bspmellintrans}(ii) we calculated $\mathcal{M}[BES_1](s) = \frac{1}{\Gamma\left({d}/{2}\right)}\Gamma\left(\frac{s+d-1}{2}\right) 2^{\frac{s-1}{2}}$. Looking at \eqref{schaetzer} we obtain
\begin{align} 
 \hat{f_n}(x)=  \frac{1}{2\pi} \int_{-g_n}^{g_n} \frac{ \Gamma\left(\frac{d}{2}\right)\frac{1}{n}\sum_{k=1}^{n}X_k^{2(\gamma-1+iv)}}{\Gamma\left(\gamma+\frac{d}{2}-1+iv\right) 2^{\gamma-1+iv}}  x^{-\gamma-iv} ~dv \label{schaetzerbessel}
\end{align}
as an estimator for the density $f_T(x)$ of a stopping time $T\geq 0$ for $x>0$ and $\max\{1/2,a\}< \gamma < b$, where $a,b$ are such that $f_T\in\mathfrak{M}_{(a,b)}$ and $X_1,...,X_n$ are independent samples of $BES_T$.
With our major result Theorem \ref{konvbessel} we shall derive the convergence rates for \eqref{schaetzerbessel}.

\begin{satz}\label{konvbessel}
If $f_T\in\mathcal{C}(\beta,a,b)$ for some $0\leq a<b$, $\beta\in(0,\pi)$ and if there is a $\gamma\in(a,b)$ with $2\gamma-1\in(a,b)$ and $\gamma > (4-d)/4$, then
\begin{equation}\label{abschinmaini}
 x^{2\gamma}\E [  |f_T(x)- \hat{f_n}(x)|^2 ]  \leq C_{L,d,\gamma}  \left(\frac{1}{n}e^{\pi g_n} + e^{-2\beta g_n}\right),\quad x> 0 
\end{equation}
for some $C_{L,d,\gamma}>0$ depending only on $L,\gamma,d$ as well as $T$.
Moreover, taking 
\begin{equation}\label{choicehninbessel}
{g_n= \frac{\log n}{\pi+2\beta}}
\end{equation}
in \eqref{abschinmaini}, one has the polynomial convergence rate
\begin{equation}\label{rateinbessel}
 \sup\limits_{x> 0}\left\{	x^{\gamma}\sqrt{\E [  |f_T(x)- \hat{f_n}(x)|^2 ]} \right\}  \lesssim  n^\frac{-\beta }{\pi+2\beta} ,\quad n\rightarrow \infty.
\end{equation}
\end{satz} 
\begin{proof}
Let $x> 0$. We use the upper bound on variance obtained in Lemma \ref{abschaetzung varianz} with $H=1/2$ to get 
\begin{align}
\Var[ x^\gamma \hat f_n(x) ] &\leq \frac{C_0(\gamma,d)}{4\pi^2 n} \left(\int_{-g_n}^{g_n}  \frac{1}{|\mathcal{M}[BES_1]({2\gamma-1+2iv})|}   dv \right)^2 \label{abschvarinproof}
\end{align} 
for some $C_0(\gamma,d)>0$. By Example \ref{bspmellintrans}(ii) and Lemma \ref{abschgamma}(ii) we have
\begin{align}
\Var[ x^\gamma \hat f_n(x) ] &\leq \frac{C_0(\gamma,d)\Gamma\left(\frac{d}{2}\right)^2 }{\pi^2 2^{2\gamma}n} \left(\int_{-g_n}^{g_n}  \frac{1}{|\Gamma\left(\gamma-1+\frac{d}{2}+iv\right)|} dv\right)^2 \nonumber\\
&\leq \frac{C_0(\gamma,d)\Gamma\left(\frac{d}{2}\right)^2 }{\pi^2 2^{2\gamma}n} (C_1(d,\gamma)+C_2e^{\pi g_n/2})^2 \label{boundonvar}
\end{align} 
for some constants $C_1(d,\gamma)$ and $C_2$. Adding \eqref{boundonvar} and \eqref{biasabschforc} gives
\begin{align}
\E [  |f_T(x)- \hat{f_n}(x)|^2 ] &\leq \frac{C_0(\gamma,d)\Gamma\left(\frac{d}{2}\right)^2 }{\pi^2 2^{2\gamma}n} (C_1(d,\gamma)+C_2e^{\pi g_n/2})^2 + C^2 e^{-2\beta g_n} \nonumber\\
&\leq C_{L,d,\gamma} \left(\frac{1}{n}e^{\pi g_n} + e^{-2\beta g_n}\right) \label{gleicheraten}
\end{align}
for some $C_{L,d,\gamma}>0$. The choice \eqref{choicehninbessel} yields the rate \eqref{rateinbessel}.
\end{proof}
The class $\mathcal{C}(\beta,a,b)$ is fairly large. In particular, $\mathcal{C}(\beta,0,\infty)$ includes for all $\beta\in(0,\pi/2)$ such well-known families of distributions as Gamma, Weibull, Beta, log-normal and inverse Gaussian. So, if $T$ belongs to one of those families, Theorem \ref{konvbessel} is true for any $\gamma>\max\{1/2,(4-d)/4 \}$. If $d\geq 2$, then we only require $\gamma>1/2$.

\section{Asymptotic Normality for Bessel Processes}\label{asymptoticnormality}
Note that the estimator \eqref{schaetzerbessel} can be written as 
\begin{equation} \label{representationassum}
\hat{f_n}(x)= \frac{1}{n}\sum_{k=1}^{n} Z_{n,k}
\end{equation}
with
\begin{equation}\label{znk}
Z_{n,k}:= \frac{\Gamma\left(\frac{d}{2}\right)}{\pi} \int_{-g_n}^{g_n} \frac{X_k^{2(\gamma-1+iv)}}{\Gamma\left(\gamma+\frac{d}{2}-1+iv\right) 2^{\gamma+iv}}  x^{-\gamma-iv} ~dv.
\end{equation}
Since $\hat{f_n}$ is a sum of iid variables, we can show that (under mild assumptions on $f_T$) $\hat{f_n}$ is asymptotically normal. In fact, we have:
\begin{satz}\label{asympnormality}
	Let $f_T\in \mathfrak{M}_{(a,b)}$ for some $0\leq a<b$. Suppose there is a $\gamma\in (a,b)$ such that $2\gamma-1\in (a,b)$, $\gamma>(4-d)/4$ and $(\delta+2)\gamma-\delta-1\in (a,b)$ for some $\delta>0$
	and
	\begin{equation}\label{mellinendl}
	\int_{-\infty}^\infty |\M[T](2\gamma-1+iv)|dv<\infty.
	\end{equation} 
	
	If we choose $g_n \sim \log(n)$ in \eqref{schaetzerbessel} then we have
	\begin{equation}\label{claimnormalit}
	\sqrt{n}\nu_n^{-1/2}(\hat{f_n}(x)-\E[\hat{f_n}(x)]) \stackrel{d}{\rightarrow} \mathcal{N}(0,1)
	\end{equation}
	
	for all $x>0$, where 
	\begin{equation}\label{varinnormal}
	\nu_n:=\Var[Z_{n,1}]=\frac{c\Gamma\left(\frac{d}{2}\right)}{2\pi^3 x^{2\gamma}} g_n^{-2\gamma+d-3} e^{\pi g_n} \log^{-2}(g_n) (1+o(1))
	\end{equation}
	with some $c>0$ given by \eqref{konstanteinasympnorm}. 
\end{satz}
We present the proof in Section \ref{proofasympnorm}. As we mentioned in the end of Section \ref{secbesselprozess}, we can often assume $(a,b)=(0,\infty)$ so that the choice of $\gamma$ is only restricted by $\gamma>\max\{1/2,(4-d)/4 \}$. If $(a,b)=(0,\infty)$, then a suitable $\delta$ can always be found, for instance, any $\delta>2\gamma/(1-\gamma)$ is valid. If additionally $d\geq 2$, then the statement is true for all $\gamma>1/2$.

It is possible to give a Berry-Esseen type error estimate for the convergence in \eqref{claimnormalit}. This is a new result even for dimension $d=1$. 

\begin{satz}\label{berryesseen}
	Let $f_T\in \mathfrak{M}_{(a,b)}$ for some $0\leq a<b$. Suppose there is a $\gamma\in (a,b)$ such that $2\gamma-1\in (a,b)$, $\gamma>(4-d)/4$, $3\gamma-2\in (a,b)$ and \eqref{mellinendl} holds.
	Fix some $x>0$. Denote by $F_n$ the distribution function of $$\sqrt{n}\nu_n^{-1/2}(\hat{f_n}(x)-\E[\hat{f_n}(x)])$$ (where $\hat f_n(x)$ is defined by \eqref{schaetzerbessel} and $\nu_n=n\Var[\hat f_n(x)]$ is given by \eqref{varinnormal}) and by $\Phi$ the distribution function of the standard normal distribution. If we choose $g_n \sim \log(n)$ in \eqref{schaetzerbessel} then we have
	\begin{equation}\label{claimberryesseen}
	\rho_n :=\sup\limits_{y\in\R} |F_n(y)-\Phi(y)|\lesssim  \begin{cases}
	n^{-\frac{1}{2}} (\log n)^{4\gamma-d+3} \log^{3}(\log(n)) , &\text{~if~}\gamma<\frac{3-d}{2}\\
	n^{-\frac{1}{2}} (\log n)^{3(2\gamma-d+3)/2} \log^{3}(\log(n))  , &\text{~if~}\gamma \geq \frac{3-d}{2}
	\end{cases}
	\end{equation}
	for $n\ra\infty$.
\end{satz}
\begin{proof}
	Let $x>0$ and $n\in\N$. Consider the representation \eqref{representationassum} of $\hat f_n(x)$. Berry-Esseen Theorem (see \cite{gnssler}) states
	\begin{equation}\label{abschforberry}
	\rho_n \leq \frac{6 \E[|Z_{n,1}-\E[Z_{n,1}]|^3]}{(\Var[Z_{n,1}])^{3/2} n^{1/2}}.
	\end{equation}
	We choose $j=3$ in Lemma \ref{momentevonzn} to get
	\begin{equation}\label{abschdritteszentrmoment}
	\E[|Z_{n,1}-\E[Z_{n,1}]|^3]\lesssim \E[|Z_{n,1}|^3]\lesssim \begin{cases}
	g_n^{-\gamma-d/2+3/2} e^{3\pi g_n/2}, &\text{~if~}\gamma+d/2-3/2 < 0\\
	e^{3\pi g_n/2}, &\text{~if~}\gamma+d/2-3/2 \geq 0
	\end{cases}
	\end{equation}
	for $n\ra\infty$. By Theorem \ref{asympnormality} we have \eqref{varinnormal}. Choose ${g_n\sim\log(n)}$. Plugging \eqref{abschdritteszentrmoment} and \eqref{varinnormal} into \eqref{abschforberry} concludes the proof.
\end{proof}
Note that the signs of the powers ${4\gamma-d+3}$ and ${3(2\gamma-d+3)/2}$ in \eqref{claimberryesseen} are ambiguous and depend on the relative positions of $\gamma$ and $d$. However, if $d\geq 2$ then we only have the case $\gamma+d/2-3/2\geq 0$ and the power of the logarithm is positive.\\
The following observation about the absolute moments of $Z_{n,1}$ is useful in the proof of Theorem \ref{asympnormality} but also holds some insights in itself.
\begin{lem}\label{momentevonzn}
	Let $f_T\in \mathfrak{M}_{(a,b)}$ for some $0\leq a<b$ and $(g_n)_{n\in\N}\subset \R_+$ with $g_n\ra\infty$ as $n\ra\infty$. If there is a $\gamma\in (a,b)$ such that $2\gamma-1\in (a,b)$, $\gamma>(4-d)/4$ and $(\gamma-1)j+1\in(a,b)$, then
	\begin{equation} \label{momentevonznabsch}
	\E[|Z_{n,1}|^{j}] \lesssim \begin{cases}
	g_n^{-\gamma-d/2+3/2} e^{\pi g_n j/2}, &\text{~if~}\gamma+d/2-3/2 < 0\\
	e^{\pi g_n j/2}, &\text{~if~}\gamma+d/2-3/2 \geq 0
	\end{cases}
	\end{equation}
	as $n\ra\infty$ for all $j\in\R_+$. In particular, all absolute moments of $Z_{n,1}$ exist for all $n\in\N$ greater than some $n_0\in\N$.
\end{lem}
\begin{proof}
	Case $\gamma+d/2-3/2 \geq 0$: By Jensen inequality, Lemma \ref{abschgamma}(ii) and \eqref{skalmitmellin2} (with $H=1/2$ and $Y=BES$ there) we have	
	\begin{align*}
	\E[|Z_{n,1}|^{j}] \leq&~ c  \E [X_1^{2(\gamma-1)j}] \int_{-g_n}^{g_n} \frac{1}{\left| \Gamma\left(\gamma+d/2-1+iv\right)\right|^{j}}  ~dv  \\
	\leq&~ c \M[X_1](2(\gamma-1)j+1) (C_{\gamma,d,j}+C_j e^{\pi g_n j/2} )\\
	=&~ c  \M[T]((\gamma-1)j+1) \M[BES_1](2(\gamma-1)j+1) (C_{\gamma,d,j}+C_j e^{\pi g_n j/2} ),
	\end{align*}
	where $c:=\Gamma\left(\frac{d}{2}\right)^{j}(2^\gamma\pi x^\gamma)^{-j}  $ and $C_{\gamma,d,j},C_j>0$. The case $\gamma+d/2-3/2 < 0$ follows similarly applying Lemma \ref{abschgamma}(i) instead of (ii).
\end{proof}
For the special case $(a,b)=(0,1)$ and $d=j=1$ this result is mentioned in \cite{Belomestnyasymnorm} but without an extensive proof which we provide here.
Note that for $d\geq 2$ the assumption $\gamma>(4-d)/4$ is redundant. Moreover, we always have the smaller bound of the second case in \eqref{momentevonznabsch}.

\section{Some Other Self-Similar Processes} \label{secotherprocesses}
\subsection{Normally Distributed Processes}\label{subsecfracbrown}
Let $Y=(Y_t)_{t\geq 0}$ be $H$-ss with càdlàg paths and $Y_1$ standard normally distributed. As example consider a fractional Brownian motion. This setting is easily generalized to the case where $Y_1\sim \mathcal N(0,\sigma^2)$ with $\sigma^2>0$ by considering the process $(\tilde Y_t)_{t\geq 0} := (Y_t/\sigma)_{t\geq 0}$ and modifying our observations to $\tilde X_i :=X_i/\sigma$.
Taking $d=1$ in Example \ref{bspmellintrans}(ii) we see that estimator \eqref{schaetzer} assumes the form
\begin{equation} \label{schaetzergausssch}
\hat{f_n}(x)=  \frac{1}{2\sqrt{\pi}} \int_{-g_n}^{g_n} \frac{ \frac{1}{n}\sum_{k=1}^{n}X_k^{(\gamma-1+iv)/H}}{\Gamma\left(\frac{\gamma+H-1+iv}{2H}\right) 2^{\frac{\gamma+2H-1+iv}{2H}}}  x^{-\gamma-iv} ~dv 
\end{equation}
for $x>0$ and $\max\{1-H,a\}< \gamma < b$. We can prove a convergence result for this estimator, similar to Theorem \ref{konvbessel}. 

\begin{satz}\label{konvgausssch}
Let $0\leq a<b$. Suppose $f_T\in\mathcal{C}(\beta,a,b)$ for some  $\beta\in(0,\pi)$. If there is some $\gamma\in(\max\{a,1-H,3/4\},b)$, then
		\begin{equation}\label{absch1}
		x^{2\gamma} \E [  |f_T(x)- \hat{f_n}(x)|^2 ]  \lesssim  \begin{cases}
		\frac{1}{n}e^{\pi g_n/(2H)} + e^{-2\beta g_n}, &\text{if} ~~\gamma\geq 1\\
		\frac{1}{n} g_n^{(1-\gamma)/H}e^{\pi g_n/(2H)}+ e^{-2\beta g_n}, &\text{if} ~~\gamma< 1
		\end{cases}
		\end{equation}
		for $n\rightarrow \infty$ and all $x>0$. Taking
		\begin{equation}\label{choicehn1}
		g_n=
		\begin{cases}
		\frac{2H\log n}{\pi+4H\beta}, &\text{if} ~~\gamma\geq 1\\
		\frac{2H\log n  - 2(\gamma-1)\log\log n}{\pi+4H\beta}, &\text{if} ~~\gamma< 1
		\end{cases},
		\end{equation}
		we obtain the polynomial convergence rate
		\begin{equation}\label{rate1}
	 \sup\limits_{x> 0}\left\{	x^{\gamma}\sqrt{\E [  |f_T(x)- \hat{f_n}(x)|^2 ]} \right\} \lesssim   \begin{cases}
		n^{-\frac{2H\beta}{\pi+4H\beta}}, &\text{if} ~~\gamma\geq 1\\
		n^{-\frac{2H\beta}{\pi+4H\beta}} (\log n)^{\frac{(1-\gamma)2\beta}{\pi+4H\beta}}, &\text{if} ~~\gamma< 1
		\end{cases}
		\end{equation}
		for $n\rightarrow \infty$.
\end{satz} 
\begin{proof}
The proof is analogous to the one of Theorem \ref{konvbessel} except for the upper bound on variance which is in this case
\begin{align*}
\Var[ x^\gamma \hat f_n(x) ]
&\leq \frac{C_0(\gamma,H) H^2 }{\pi 2^{\frac{\gamma+2H-1}{2H}}n} \times
\begin{cases}
(C_1(\gamma,H)+C_2e^{\frac{\pi g_n}{4H}})^2, &\text{if} ~~\gamma\geq 1\\
C_1^2\left(2H/g_n\right)^{\frac{\gamma-1}{H}}e^{\frac{\pi g_n}{2H}}, &\text{if} ~~\gamma< 1
\end{cases}
\end{align*} 
for some $C_0(\gamma,H)>0$. Combining this with the bound on the bias from Lemma \ref{abschaetzung erwartungswert fuer klasse}(i) we obtain \eqref{absch1}. Plugging \eqref{choicehn1} into \eqref{absch1} gives the rate \eqref{rate1}.
\end{proof}
Taking $H=1/2$ in Theorem \ref{konvgausssch} we obtain the same rates as for Bessel processes (see Theorem \ref{konvbessel}). For smaller $H$ the rate is worse and for greater $H$ it is better. Note that we work with observations of $|Y_T|$ rather than $Y_T$.

\subsection{Gamma Distributed Processes}\label{subsecsquaredbessel}
Let $Y=(Y_t)_{t\geq 0}$ be $H$-ss with càdlàg paths such that $Y_1$ has Gamma density \eqref{gammadichte} with $r=1$. We can easily generalize to the case $r>0$, by considering the process ${(\tilde Y_t)_{t\geq 0} := (r Y_t)_{t\geq 0}}$ and modifying our observations to $\tilde X_i :=r X_i$. As an example consider the so-called \textit{square of a Bessel process} with dimension $d$ starting at $0$ (see \cite[Chapter XI, §1]{yor}). Considering Example \ref{bspmellintrans}(i) estimator \eqref{schaetzer} takes the form
\begin{equation} \label{schaetzergamma}
\hat{f_n}(x)=  \frac{\Gamma(\sigma)}{2{\pi}} \int\limits_{-g_n}^{g_n} \frac{ \frac{1}{n}\sum_{k=1}^{n}X_k^{(\gamma-1+iv)/H}}{\Gamma\left(\frac{\sigma H+\gamma-1+iv}{H} \right)}  x^{-\gamma-iv} ~dv 
\end{equation}
for $x>0$ and $\max\{1-\sigma H ,a\}< \gamma < b$. We can prove a convergence result for this estimator, that is similar to Theorems \ref{konvbessel} and \ref{konvgausssch}.

\begin{satz}\label{konvgamma}
Let $0\leq a<b$. Suppose $f_T\in\mathcal{C}(\beta,a,b)$ for some  $\beta\in(0,\pi)$. If there is some $\gamma\in(\max\{a,1-\sigma H, 1-\sigma/4 \},b)$ with $2\gamma-1\in(a,b)$, then
		\begin{equation}\label{abschgamma1}
		x^{2\gamma} \E [  |f_T(x)- \hat{f_n}(x)|^2 ]  \lesssim  \begin{cases}
		\frac{1}{n} e^{\frac{\pi g_n}{H}}+ e^{-2\beta g_n}, &\text{if} ~~\gamma\geq 1-\sigma H+\frac{H}{2}\\
		\frac{1}{n} g_n^{1-\frac{2(\gamma+\sigma H-1)}{H}}e^{\frac{\pi g_n}{H}}+ e^{-2\beta g_n},&\text{if} ~~\gamma< 1-\sigma H+\frac{H}{2}
		\end{cases}
		\end{equation}
		for $n\rightarrow \infty$ and all $x > 0$. If
		\begin{equation}\label{choicehngamma1}
		g_n=
		\begin{cases}
		\frac{\log n}{\frac{\pi}{H}+2\beta}, &\text{if} ~~\gamma\geq 1-\sigma H+\frac{H}{2}\\
		\frac{\log n  - \left(1-\frac{2(\gamma+\sigma H-1)}{H}\right)\log\log n}{\frac{\pi}{H}+2\beta}, &\text{if} ~~\gamma< 1-\sigma H+\frac{H}{2}
		\end{cases},
		\end{equation}
	then
		\begin{equation}\label{rategamma1}
\sup\limits_{x> 0}\left\{	x^{\gamma} \sqrt{\E [  |f_T(x)- \hat{f_n}(x)|^2 ]}\right\} \lesssim  \begin{cases}
		n^{-\frac{\beta}{\frac{\pi}{H}+2\beta}}, &\text{if} ~\gamma\geq 1-\sigma H+\frac{H}{2}\\
		n^{-\frac{\beta}{\frac{\pi}{H}+2\beta}} (\log n)^{-k}, &\text{if} ~\gamma< 1-\sigma H+\frac{H}{2}
		\end{cases}
		\end{equation}
		for $n\rightarrow \infty$, where $k=\frac{\beta}{\frac{\pi}{H}+2\beta} \left(\frac{2(\gamma+\sigma H-1)}{H}-1\right)$.
\end{satz} 
\begin{proof}
In this case he upper bound on variance becomes
\begin{align*}
\Var[ x^\gamma \hat f_n(x) ] &\lesssim  
\begin{cases}
\frac{1}{n} e^{\frac{\pi g_n}{H}}, &\text{if} ~~\gamma\geq 1-\sigma H+\frac{H}{2}\\
\frac{1}{n} g_n^{1-\frac{2(\gamma+\sigma H-1)}{H}}e^{\frac{\pi g_n}{H}}, &\text{if} ~~\gamma< 1-\sigma H+\frac{H}{2}
\end{cases}, 
\end{align*} 
Rest is again analogue to the proof of Theorem \ref{konvbessel}.
\end{proof}

\section{Optimality}\label{optimalitysection}
The rates from Theorems \ref{konvbessel}, \ref{konvgausssch} and \ref{konvgamma} are optimal in the minimax sense.
\begin{satz}\label{optimalitysatz}
	For all $\beta\in(0,\pi)$ and $0<a<b<\pi/\beta$ there is $x>0$ such that
	\begin{equation}\label{minimaxinequality}
	\liminf\limits_{n\ra\infty} \psi_n^{-2} \inf\limits_{\hat f_n} \sup\limits_{f\in\mathcal C(\beta,a,b)} \E[|\hat f_n(x)-f(x)|^2] \geq c
	\end{equation}
	for some $c>0$, where infimum is over all estimators based on samples of $Y_T$ with 
	\begin{enumerate}
		\item[(i)] a Bessel process $Y$ with dimension $d\in [1,\infty)$ and $\psi_n=n^{-\frac{\beta}{\pi+2\beta}}$;
		\item[(ii)] a $H$-ss. Gaussian process $Y$ ($H\in(0,2)$) and $\psi_n=n^{-\frac{\beta}{\frac{\pi}{2H}+2\beta}}$;
		\item[(iii)] a $H$-ss. Gamma distributed process $Y$ ($H\in(0,2)$) and $\psi_n=n^{-\frac{\beta}{\frac{\pi}{H}+2\beta}}$.
	\end{enumerate}
\end{satz}
See Section \ref{secoptimalityproof} for the proof of this theorem. A similar optimality result was obtained in \cite{Belomestnyasymnorm} for the case where the absolute value of a one-dimensional Brownian motion is observed. \eqref{minimaxinequality} means that for each estimator $\hat f_n$, that we may construct with our observations, there is a true density $f\in\mathcal C(\beta,a,b)$ such that
$$ \sqrt{\E[|\hat f_n(x)-f(x)|^2]}\gtrsim \psi_n,\quad n\ra\infty $$
for some $x>0$, i.e. it is impossible to construct an estimator with a convergence rate (w.r.t. $L^2$-distance) faster than $\psi_n$ for all $f\in\mathcal C(\beta,a,b)$ and all $x>0$.

\section{Simulation Study}\label{simulationseasy}
In this Section we test our estimator \eqref{schaetzerbessel} with some simulated data. Consider a Bessel process with dimension $d=5$ and a Gamma$(2,1)$ distributed stopping time $T$, i.e. $T$ has the density
$$f(x)=xe^{-x}, \quad x\geq 0.$$
In order to evaluate the estimator \eqref{schaetzerbessel} we choose $\gamma=0.7$. Take the cut-off parameter $g_n=\frac{\log(n)}{(\pi+2\beta)}$ (in accordance with \eqref{choicehninbessel}) and $\beta=0$. To choose $\beta$ small appears counterintuitive at first because we showed in Theorem \ref{konvbessel} that the convergence rate is better for large $\beta$. However, in our examples the choice $\beta=0$ delivers the best results. This can be explained as follows: Our bound on the bias of estimator $\hat f_n$ contains the constant $L$ (see \eqref{Lkonstante}) as a factor. This constant is growing in $\beta$ and seems to make a crucial contribution to the overall error. We refer to \cite{Belomestnyasymnorm} and \cite{diss} for an alternative choice of $g_n$ based purely on the data.\\In order to test the performance of $\hat f_n$ we compute it based on 100 independent samples of $BES_T$ of size $n\in\{1000,5000,10000,50000\}$. In Figure \ref{fig2} we see the resulting box-plots of the loss.
\begin{figure}[h!]
	\centering
	\includegraphics[width=0.6\linewidth,height=0.3\textheight]{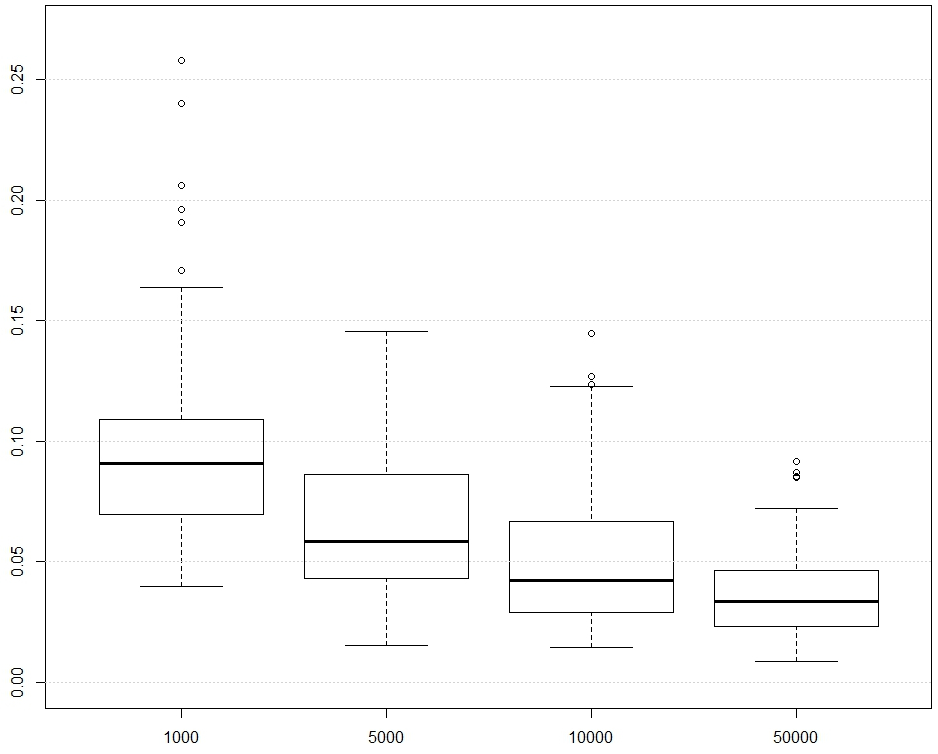}
	\captionsetup{width=\linewidth}
	\caption{Box plot of the loss $\sup_{x\in [0,10]} \{|\hat f_n(x)-f_T(x)| \}$ for different sample sizes.}
	\label{fig2}
\end{figure}
Let us now demonstrate the performance of our estimator for different distributions of $T$. As examples we consider Exponential, Gamma, Inverse-Gaussian and Weibull distributions. To construct the estimate \eqref{schaetzerbessel} we choose $d=5$, $\gamma=0.8$, $n=1000$ and $g_n$ as before. Figure \ref{fig3} shows the densities of the six distributions and their 50 respective estimates based on 50 independent samples of $BES_T$ of size $n=500$.
\begin{figure}[h!]
	\centering
	\includegraphics[width=0.9\linewidth,height=0.3\textheight]{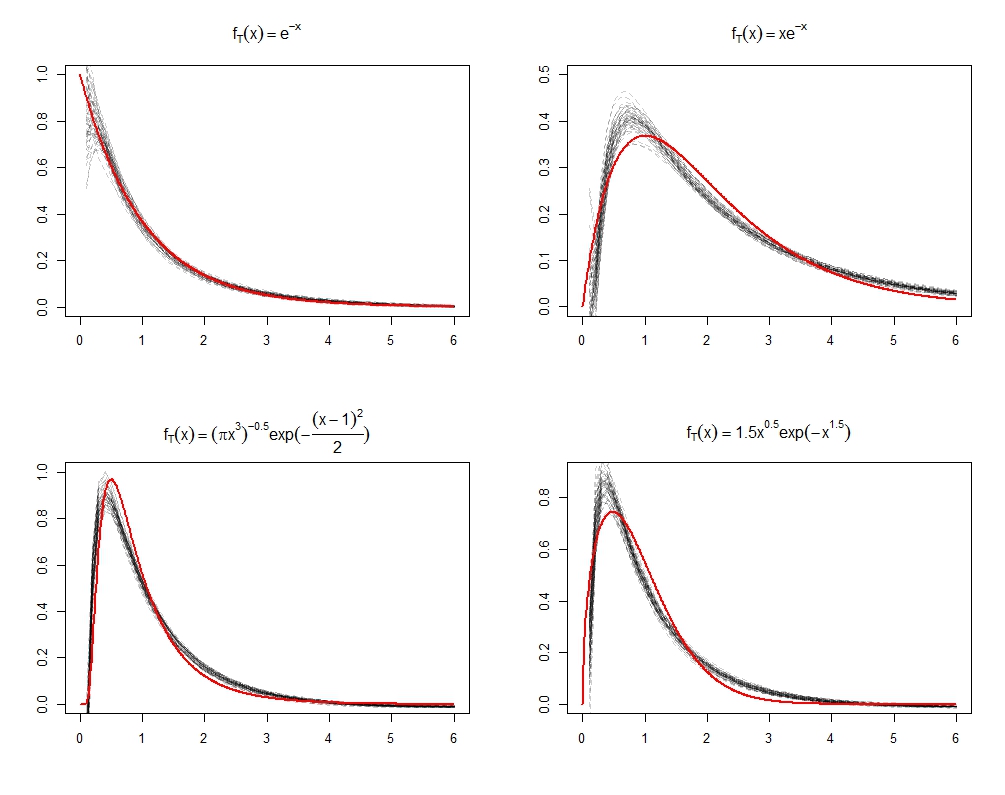}
	\caption{Estimated densities (red) and their 50 respective estimates (grey).}
	\label{fig3}
\end{figure}

\section{Proofs}\label{proofsec}
\subsection{Proof of Theorem \ref{asympnormality}}\label{proofasympnorm}
	We roughly imitate the proof of an analogous result for the special case ${d=1}$, $(a,b)=(0,1)$ found in \cite{Belomestnyasymnorm}. In distinction from \cite{Belomestnyasymnorm} we do not restrict ourselves to the case $x=1$ in the proof and provide the specific form of $\nu_n$ for all $x>0$.\\
	Let $x>0$. It suffices to show the Lyapunov condition, i.e. for a $\delta>0$:
	\begin{equation}\label{lyapunovcond}
	\lim\limits_{n\rightarrow \infty} \frac{\E[|Z_{n,1}-\E[ Z_{n,1}]|^{2+\delta}]}{n^{\delta/2}(\Var[Z_{n,1}])^{1+\delta/2}} =0.
	\end{equation}
	The claim \eqref{claimnormalit} follows from \eqref{lyapunovcond} with $\nu_n=\Var[Z_{n,1}]$. Note that $\E [Z_{n,1}]\ra f_T(x)$ for $n\ra\infty$ by monotone convergence and \eqref{exactsolforbessel} (if we choose $Y=BES$ there). So, \eqref{lyapunovcond} holds if we can prove, that $\Var[Z_{n,1}]\ra\infty$ and
	\begin{equation}\label{lyapunovcondsimple}
	\lim\limits_{n\rightarrow \infty} \frac{\E[|Z_{n,1}|^{2+\delta}]}{n^{\delta/2}(\Var[Z_{n,1}])^{1+\delta/2}} = 0.
	\end{equation}
	In any case of Lemma \ref{momentevonzn} (for $j=\delta+2$) we have
	\begin{equation} \label{abschez}
	\E[|Z_{n,1}|^{2+\delta}] \lesssim
	g_n^{c} e^{\pi (2+\delta)g_n/2}, \quad n\ra\infty
	\end{equation}
	for all $\delta\in\R_+$ and some $c>0$. Now we investigate the asymptotic behavior of $\Var[Z_{n,1}]$. Looking at \eqref{znk} we use Fubini's theorem to obtain
	\begin{align*}
	\Var[Z_{n,1}]=&~ \frac{\Gamma\left(\frac{d}{2}\right)^2}{\pi^2} \int_{-g_n}^{g_n}\int_{-g_n}^{g_n} \frac{\Cov[X_1^{2\gamma-1+iv},X_1^{2\gamma-1+iu}]dvdu}{(2x)^{2\gamma+i(v-u)}\Gamma\left(\gamma+\frac{d}{2}-1+iv\right)\Gamma\left(\gamma+\frac{d}{2}-1-iu\right) }   \\
	=&~ \frac{\Gamma\left(\frac{d}{2}\right)^2}{\pi^2} \int_{-g_n}^{g_n}\int_{-g_n}^{g_n} \frac{\E[X_1^{4\gamma-4+2i(v-u)}]dvdu}{(2x)^{2\gamma+i(v-u)}\Gamma\left(\gamma+\frac{d}{2}-1+iv\right)\Gamma\left(\gamma+\frac{d}{2}-1-iu\right) }   \\
	&- \frac{\Gamma\left(\frac{d}{2}\right)^2}{\pi^2} \int_{-g_n}^{g_n}\int_{-g_n}^{g_n} \frac{ (2x)^{-2\gamma-i(v-u)}\E [X_1^{2\gamma -2+ 2iv}]\E [X_1^{2\gamma-2 - 2iu}]dvdu }{\Gamma\left(\gamma+\frac{d}{2}-1+iv\right)\Gamma\left(\gamma+\frac{d}{2}-1-iu\right) }  \\
	=&~ \frac{\Gamma\left(\frac{d}{2}\right)^2}{\pi^2} \int_{-g_n}^{g_n}\int_{-g_n}^{g_n} \frac{\M[X_1](4\gamma+2i(v-u)-3)dvdu}{(2x)^{2\gamma+i(v-u)}\Gamma\left(\gamma+\frac{d}{2}-1+iv\right)\Gamma\left(\gamma+\frac{d}{2}-1-iu\right) }   \\
	&-\frac{\Gamma\left(\frac{d}{2}\right)^2}{\pi^2} \left|\int_{-g_n}^{g_n} \frac{\M[X_1](2\gamma -1+ 2iv) }{(2x)^{\gamma+iv}\Gamma\left(\gamma+\frac{d}{2}-1+iv\right)} dv \right|^2~=:R_1-R_2.
	\end{align*}
	By Example \ref{bspmellintrans}(ii) we can estimate
	\begin{align*}
	R_2 &\leq \frac{1}{\pi^2 x^{2\gamma}} \left(\int_{-g_n}^{g_n} \left| \M[T](\gamma + iv) \right| dv \right)^2 <C< \infty
	\end{align*}
	for some $C>0$ and further
	\begin{align*}
	R_1 &=  \frac{\Gamma\left(\frac{d}{2}\right)}{2\pi^2 x^{2\gamma}} \int_{-g_n}^{g_n}\int_{-g_n}^{g_n} \frac{\M[T](2\gamma-1+i(v-u)) \Gamma\left(2\gamma-2+\frac{d}{2}+i(v-u)\right) }{x^{i(v-u)}\Gamma\left(\gamma+\frac{d}{2}-1+iv\right)\Gamma\left(\gamma+\frac{d}{2}-1-iu\right) } dvdu.
	\end{align*}
	Our strategy now is to decompose the double integral defining $R_1$ into pieces that are easy to estimate. To that end let $\rho_n:=g_n^{\alpha}$, where $0<\alpha<1/2$ and define
	$$I^1_{n}:= \int_{-g_n}^{g_n}\int_{-g_n}^{g_n} \1_{|v-u|\geq \rho_n} \frac{\M[T](2\gamma-1+i(v-u)) \Gamma\left(2\gamma-2+\frac{d}{2}+i(v-u)\right) }{x^{i(v-u)}\Gamma\left(\gamma+\frac{d-2}{2}+iv\right)\Gamma\left(\gamma+\frac{d-2}{2}-iu\right) } dvdu.$$
	By Lemma \ref{abschgamma0} there are $C_1,C_2>0$ such that
	\begin{align*}
	\left|\Gamma\left(\gamma+(d/2)-1+iv\right)\right| &\geq C_1 \1_{|v|\leq 2 }+ C_2 \1_{|v|>2 } |v|^{\gamma-1+\frac{d-1}{2}} e^{-\pi |v|/2}\\
	\left|\Gamma\left(\gamma+(d/2)-1-iu\right)\right| &\geq  C_1 \1_{|u|\leq 2 }+ C_2 \1_{|u|>2 } |u|^{\gamma-1+\frac{d-1}{2}} e^{-\pi |u|/2}
	\end{align*}
	and $K_1,K_2>0$ such that
	\begin{align*}
	\Gamma\left(2\gamma-2+(d/2)+i(v-u)\right) \leq  K_1 \1_{|v-u|\leq 2 }+ K_2 \1_{|v-u|>2 } |v-u|^{2(\gamma-1)+\frac{d-1}{2}} e^{-\pi |v-u|/2}.
	\end{align*}
	With the help of these inequalities we deduce 
	\begin{align*}
	|I^1_{n}|
	\lesssim  g_n^{3|\gamma-1|+d-2} e^{\pi(g_n-\rho_n)/2} +g_n^{4|\gamma-1|+(3d-5)/2} e^{\pi(g_n-\frac{\rho_n}{2})}, \quad n\ra\infty.
	\end{align*}
	Similarly,
	\begin{align}
	\int_{-g_n}^{g_n}&\int_{-g_n}^{g_n} \1_{|u|\leq g_n-\rho_n}  \1_{|v-u|\geq \rho_n} \frac{\M[T](2\gamma-1+i(v-u)) \Gamma\left(2\gamma+\frac{d-4}{2}+i(v-u)\right) }{x^{i(v-u)}\Gamma\left(\gamma+\frac{d-2}{2}+iv\right)\Gamma\left(\gamma+\frac{d-2}{2}-iu\right) } dvdu \nonumber\\
	&\lesssim  g_n^{l} e^{\pi(g_n-\rho_n)} , \quad n\ra\infty \label{teil1vonI2}
	\end{align}
	and
	\begin{align}
	\int_{-g_n}^{g_n}&\int_{-g_n}^{g_n} \1_{|v|\leq g_n-\rho_n}  \1_{|v-u|\geq \rho_n} \frac{\M[T](2\gamma-1+i(v-u)) \Gamma\left(2\gamma+\frac{d-4}{2}+i(v-u)\right) }{x^{i(v-u)}\Gamma\left(\gamma+\frac{d-2}{2}+iv\right)\Gamma\left(\gamma+\frac{d-2}{2}-iu\right) } dvdu\nonumber\\
	&\lesssim  g_n^{l} e^{\pi(g_n-\rho_n)} , \quad n\ra\infty \label{teil2vonI2}.
	\end{align}
	for some $l\geq 0$. Combine \eqref{teil1vonI2} and \eqref{teil2vonI2} to obtain
	\begin{align*}
	I^2_{n} :=&~  \int_{-g_n}^{g_n}\int_{-g_n}^{g_n} \1_{|v-u|\leq \rho_n} \frac{\M[T](2\gamma-1+i(v-u)) \Gamma\left(2\gamma+\frac{d-4}{2}+i(v-u)\right) }{x^{i(v-u)}\Gamma\left(\gamma+\frac{d-2}{2}+iv\right)\Gamma\left(\gamma+\frac{d-2}{2}-iu\right) } dvdu  \\
	\lesssim&~  \int_{-g_n}^{g_n}\int_{-g_n}^{g_n} \1_{|v|\geq g_n-\rho_n}\1_{|u|\geq g_n-\rho_n} \1_{|v-u|\leq \rho_n}\\
	&\times \hspace*{-0.1cm} \frac{\M[T](2\gamma-1+i(v-u)) \Gamma\left(2\gamma+\frac{d-4}{2}+i(v-u)\right) }{x^{i(v-u)}\Gamma\left(\gamma+\frac{d-2}{2}+iv\right)\Gamma\left(\gamma+\frac{d-2}{2}-iu\right) } dvdu +\mathcal{O}( g_n^{l} e^{\pi(g_n-\rho_n)})\\
	=:&~I^3_{n}+\mathcal{O}( g_n^{l} e^{\pi(g_n-\rho_n)}).
	\end{align*}
	Next, we examine the asymptotic behavior of the integral $I^3_{n}$. To this end, we take advantage of Stirling's formula (Lemma \ref{gammaarggegeninf})
	\begin{align*}
	\Gamma\left(\gamma+\frac{d-2}{2}+iv\right)&= \left(\gamma+\frac{d-2}{2}+iv\right)^{\gamma+\frac{d-3}{2}+iv}e^{-\gamma-\frac{d-2}{2}-iv}\sqrt{2\pi}(1+\mathcal{O}(v^{-1}))
	\end{align*}
	for $v\ra\infty$. Consider the integrand of $I^3_{n}$. In the denominator it holds by means of the identity ${\log(iv)=\log(v)+\frac{i\pi}{2}}$ that
	\begin{align*}
	\Gamma (\gamma+(d/2)-1+iv)\Gamma(\gamma&+(d/2)-1-iu)\\
	&= 2\pi e^{iv\log v - iu\log u - i(v-u)}  
	e^{-\frac{\pi}{2}(u+v)}  (vu)^{\frac{2\gamma+d-3}{2}} (1+\mathcal{O}(v^{-1})+\mathcal{O}(u^{-1}))
	\end{align*}
	for $u,v\ra\infty$. On the set
	$$ \left\{ |u|\geq g_n-\rho_n \right\} \cap \left\{ |v|\geq g_n-\rho_n \right\} \cap \{ |v-u|\leq \rho_n \}  \cap \{ v\geq 0, u\geq 0 \} $$
	we define $u=g_n -r$, $v=g_n -s$ with $0<r,s<\rho_n$, $|r-s|<\rho_n$ to obtain
	\begin{align*}
	\Gamma (\gamma+(d&/2)-1+iv)\Gamma(\gamma+(d/2)-1-iu)\\
	=&~ 2\pi\exp[i(g_n-s)\log(g_n-s)- i(g_n-r)\log(g_n-r) -i(r-s)]
	\\&\times g_n^{2\gamma+d-3} e^{-\pi g_n} e^{\pi(r+s)/2} (1+\mathcal O(g_n^{-1}))(1+\mathcal O(\rho_n g_n^{-1})+\mathcal O(\rho_n^2 g_n^{-1}))^{\gamma+(d-3)/2}.
	\end{align*}
	Note that due to the choice of $\rho_n$, we have $\rho_n g_n^{-1}\ra 0$ and $\rho^2_n g_n^{-1}\ra 0$. We use the asymptotic decomposition 
	\begin{align*}
	(g_n -s)\log(g_n-s)- (g_n -r)\log(g_n-r)-(r-s)=(r-s)\log(g_n)+\mathcal{O}(\rho_n^2/g_n)
	\end{align*}
	to obtain
	\begin{align*}
	\Gamma\left(\gamma+(d/2)-1+iv\right)\Gamma\left(\gamma+(d/2)-1-iu\right)
	=&~2\pi g_n^{2\gamma+d-3}\exp(i(r-s)\log(g_n))\\
	&\times e^{-\pi g_n} e^{\pi(r+s)/2} (1+ O(\rho_n^2 g_n^{-1})).
	\end{align*}
	Analogously, on the set
	$$ \left\{ |u|\geq g_n-\rho_n \right\} \cap \left\{ |v|\geq g_n-\rho_n \right\} \cap \{ |v-u|\leq \rho_n \}  \cap \{ v\leq 0, u\leq 0 \} $$
	we define $u=-g_n +r$, $v=-g_n +s$ with $0<r,s<\rho_n$, $|r-s|<\rho_n$ to obtain
	\begin{align*}
	\Gamma\left(\gamma+(d/2)-1+iv\right)\Gamma\left(\gamma+(d/2)-1-iu\right)
	=&~ 2\pi g_n^{2\gamma+d-3}\exp(-i(r-s)\log(g_n))\\
	&\times e^{-\pi g_n} e^{\pi(r+s)/2} (1+ O(\rho_n^2 g_n^{-1})).
	\end{align*}
	Hence, $I^3_{n}$ can be decomposed as follows: 
	\begin{align*}
	I^3_{n} =&~ \frac{1}{2\pi}\frac{\exp(\pi g_n)}{g_n^{2\gamma+d-3}} \int_{0}^{\rho_n} \int_{0}^{\rho_n} \1_{|r-s|\leq \rho_n} x^{i(s-r)} e^{-\frac{\pi}{2}(r+s)} \Gamma(2\gamma+\frac{d-4}{2}+i(r-s))  \\
	&\times \M[T](2\gamma-1+i(r-s)) \exp(i(s-r)\log(g_n)) (1+O(\rho_n^2 g_n^{-1}))^{-1}drds \\
	=&~ \frac{1}{2\pi}\frac{\exp(\pi g_n)}{g_n^{2\gamma+d-3}} \{\Re[I^4_{n}]+O(\rho_n^2 g_n^{-1})] \},
	\end{align*}
	where
	\begin{align*}
	I^4_{n}
	:=&~ \int_{0}^{\rho_n} \int_{0}^{\rho_n} \1_{|r-s|\leq \rho_n} e^{-\frac{\pi}{2}(r+s)} \Gamma(2\gamma+\frac{d-4}{2}+i(r-s)) \\
	&\times \M[T](2\gamma-1+i(r-s)) \exp(i(s-r)\log(g_n)) x^{i(r-s)} drds \\
	=&~ \int_{0}^{\rho_n} e^{-\pi v}R_n(v)dv
	\end{align*}
	with
	\begin{equation}\label{fourierintegraltype}
	R_n(v) := \int_0^{\rho_n-v} e^{-\frac{\pi}{2} u}x^{iu} \Gamma(2\gamma+\frac{d-4}{2}+iu) \M[T](2\gamma-1+iu) e^{iu\log(g_n)}du.
	\end{equation}
	The integral in \eqref{fourierintegraltype} allows a series representation via Lemma \ref{erdelyi}. In fact,
	\begin{align*}
	R_n(v)=&~	i  \Gamma\left(2\gamma+\frac{d-4}{2}\right) \M[T](2\gamma-1) \log^{-1}(g_n) \\
	&-\frac{d}{du} [e^{-\frac{\pi}{2} u}x^{iu} \Gamma\left(2\gamma+\frac{d-4}{2}+iu\right) \M[T](2\gamma-1+iu)]\bigg|_{u=0}\log^{-2}(g_n) \\
	&+\O(\log^{-3}(g_n))  
	\end{align*} 
	uniformly in $v$. Thus,
	\begin{align*}
	\Re[I^4_{n}]= \frac{c}{\pi}\log^{-2}(g_n)+\O(\log^{-3}(g_n))
	\end{align*}
	holds with
	\begin{align}
	c:&= \Re\left[ \frac{d}{du} e^{-\frac{\pi}{2} u}x^{iu} \Gamma(2\gamma+\frac{d-4}{2}+iu)\M[T](2\gamma-1+iu)\bigg|_{u=0}\right] \nonumber\\
	& = \frac{\pi}{2} \Gamma(2\gamma+\frac{d-4}{2}) \M[T](2\gamma-1).\label{konstanteinasympnorm}
	\end{align}
	Summing up the auxiliary quantities introduced above we get
	\begin{align*}
	\Var[Z_{n,1}]=&~  \frac{\Gamma\left(\frac{d}{2}\right)}{2\pi^2 x^{2\gamma}} \left[g_n^{-2\gamma+d-3} e^{\pi g_n} \{\frac{c}{\pi}\log^{-2}(g_n)+\O(\log^{-3}(g_n))+\O(\rho_n^2 g_n^{-1}) \}   \right.\\ 
	&+ \left. \O(g_n^{l}e^{\pi(g_n-\rho_n)})\right] +\O(1)
	\end{align*}
	and thus \eqref{varinnormal}. If $g_n\sim \log(n)$, then \eqref{varinnormal} and \eqref{abschez} imply \eqref{lyapunovcondsimple} and hence the claim.

\subsection{Proof of Theorem \ref{optimalitysatz}}\label{secoptimalityproof}
	The basic construction used in this proof is due to \cite{Belomestnyallgemein}, where it is used in the context of an observed Brownian motion. Define the $\chi^2$-divergence 
	\begin{equation}\label{chisquaredcrit}
	\chi^2(P_{1}|P_{0}):=\chi^2(p_{1,M}|p_{0,M}):=\int \frac{(q_{0}(x)-q_{1}(x))^2}{q_{0}(x)}dx 
	\end{equation}
	between two probability measures $P_0$ and $P_1$ with densities $q_0$ and $q_1$. The following general result forms the basis for the subsequent steps (see \cite{tsybakov} for a proof).
	
	\begin{satz}\label{reductiontotwohypos}
		Let $\{P_f|f\in\Theta\}$ be family of probability measures indexed by a non-parametrical class of densities $\Theta$. Suppose that $X_1,...,X_n$ are iid observations in model $n$ with $\L(X_1)\in\{P_f|f\in\Theta\}$. If there are $f_{n,0},f_{n,1}\in \Theta$ such that
		\begin{equation}\label{dxkrit}
		|f_{0,n}(x)-f_{1,n}(x)| \gtrsim \psi_n, \quad n\ra\infty
		\end{equation}
		and if 
		\begin{equation}\label{modifiedchisquredcrit}
		(1+\chi^2(P_{f_{1,n}}|P_{f_{0,n}}))^n \leq \alpha
		\end{equation}
		holds for some $\alpha>0$ independent of $n$, then
		\begin{equation}
		\liminf\limits_{n\ra\infty} \psi_n^{-2} \inf\limits_{\hat f_n} \sup\limits_{f\in\Theta} \E[|\hat f_n(x)-f(x)|^2] \geq c
		\end{equation}
		holds for some $c>0$, where the infimum is over all estimators.
	\end{satz}
	
	Let $\beta\in(0,\pi)$ and $0<a<b<\pi/\beta$. Define for $M>0$
	$$q(x):=\frac{ \sin(\beta)}{\beta} \frac{1}{1+x^{\pi/\beta}} \quad
	\text{and} \quad 
	\rho_M(x):=\frac{1}{\sqrt{2\pi}} e^{-\log^2(x)/2}\frac{\sin(M\log(x))}{x}, \quad x\geq 0.$$
	The following lemma provides some properties of the functions $q$ and $\rho_M$.
	\begin{lem}
		The function $q$ is a probability density on $\R_+$ with Mellin transform
		\begin{equation}\label{mellinvonq}
		\M[q](z)=\frac{\sin(\beta)}{\sin(\beta z)},\quad 0<\Re(z)<\pi/\beta.
		\end{equation}
		The Mellin transform of the function $\rho_M$ is given by 
		\begin{equation}\label{mellinvonrhoM}
		\M[\rho_M](z)=\frac{e^{(z-1+iM)^2/2} - e^{(z-1-iM)^2/2}}{2i}, \quad z\in\C.
		\end{equation}
	\end{lem}
	\begin{proof}
		Formula \eqref{mellinvonq} can be found in \cite{mellintables} and \eqref{mellinvonrhoM} is shown in \cite[Lemma 6.2]{Belomestnyallgemein}. 
	\end{proof}
	Set now for any $M >0$ and some $\delta>0$,
	\begin{equation} \label{deffi}
	f_{0,M}(x):=q(x)\quad \text{and} \quad  f_{1,M}(x):=q(x)+ \delta(q\odot \rho_M)(x)
	\end{equation}
	for $x\geq 0$, where $q\odot \rho_M$ is defined by \eqref{mellinfaltungfkt}. The following lemma will help us verify condition \eqref{dxkrit}.
	\begin{lem}\label{supdistf0f1}
		For any $M>0$ and some $\delta>0$ not depending on $M$ the function $f_{1,M}$ is a probability density satisfying
		\begin{equation}\label{nachunteninl1}
		\sup\limits_{x\geq 0} |f_{0,M}(x)-f_{1,M}(x)|\gtrsim \exp(-M\beta),\quad M\ra\infty. 
		\end{equation} 
		Moreover, $f_{0,M}$ and $f_{1,M}$ are in $\mathcal{C}(\beta,a,b)$ for all $\beta\in(0,\pi)$ and $0<a<b<\pi/\beta$.
	\end{lem}
	\begin{proof}
		For \eqref{nachunteninl1} see \cite[Lemma 6.3]{Belomestnyallgemein} where it is also shown that for $\delta$ small enough:
		\begin{equation}\label{abschforclass}
		\delta|(q\odot \rho_M)(x)| \leq \delta\int_0^\infty \left|\frac{q(t)\rho_M(x/t)}{t} \right|dt \leq f_{0,M}(x), \quad x\geq 0. 
		\end{equation} 
		It is easy to see that $f_{0,M}\in\mathcal{C}(\beta,a,b)$ for all $\beta\in(0,\pi),~0<a<b<\pi/\beta$ and \eqref{abschforclass} implies the same for $f_{1,M}$.
	\end{proof}
	Looking further towards applying Theorem \ref{reductiontotwohypos} let us consider the densities $p_{M,0}$ and $p_{M,1}$ of an observation associated with the hypotheses $f_{0,M}$ and $f_{1,M}$, respectively. At this point we have to differentiate between the models we discussed so far. We will only present the proof for the Bessel case, parts (ii) and (iii) of Theorem \ref{optimalitysatz} can be showed along the same lines.
	
	Let $T_{0,M}$ and $T_{1,M}$ be two random variables with respective densities $f_{0,M}$ and $f_{1,M}$. The density of the random variable $BES_{T_{i,M}}$, $i=0,1$ is obtained via \eqref{dichtevonyT}:
	$$ p_{i,M}(x)=\frac{2^{1-\frac{d}{2}}}{\Gamma\left({d}/{2}\right)} x^{d-1} \int_0^\infty   \lambda^{-d/2} e^{-\frac{x^2}{2\lambda}} f_{i,M}(\lambda) d\lambda,\quad x>0,d\geq 1,i=0,1.$$
	For the Mellin transform of $p_{i,M}$ we use self-similarity of $BES$ and \eqref{besselmellin} to get
	\begin{align}
	\M[p_{i,M}](s)&=\M[BES_1](s)\M[T_{i,M}]((s+1)/2) \nonumber\\
	&= \frac{1}{\Gamma\left({d}/{2}\right)}\Gamma\left(\frac{s+d-1}{2}\right) 2^{\frac{s+1}{2}}\M[f_{i,M}]((s+1)/2), \quad i=0,1 \label{mellinvonpi}
	\end{align}
	for $\Re(s)>1-d$ and $\Re(s)\in(-1,\frac{2\pi}{\beta}+1)$.
	\begin{lem} \label{chi2distance1}
		For all $d\geq 1$ and $\beta\in(0,\pi)$ we have
		$$ \chi^2(p_{1,M}|p_{0,M})\lesssim M^{(\pi/\beta)+d-2} e^{-M(\pi+2\beta)},\quad M\ra\infty $$
	\end{lem}
	\begin{proof}
		Define $c_{\beta,d}:=\frac{2^{1-\frac{d}{2}}}{\Gamma\left({d}/{2}\right)}\frac{ \sin(\beta)}{\beta}$. By the change of variables $y=\frac{1}{\lambda}$,
		\begin{align}
		p_{0,M}(x)&= c_{\beta,d} x^{d-1} \int_0^\infty   \lambda^{-d/2} e^{-\frac{x^2}{2\lambda}}  \frac{1}{1+\lambda^{\pi/\beta}} d\lambda \nonumber\\
		&= c_{\beta,d} x^{d-1}\int_0^\infty    e^{-y\frac{x^2}{2}} y^{\frac{\pi}{\beta}+\frac{d}{2}-2} \left(1-\frac{y^{\pi/\beta}}{y^{\pi/\beta}+1}\right) dy\nonumber\\
		&= c_{\beta,d} x^{d-1}\left(\int_0^\infty    e^{-y\frac{x^2}{2}} y^{\frac{\pi}{\beta}+\frac{d}{2}-2} dy -R\right)\nonumber\\
		&= c_{\beta,d} \left(  \Gamma\left(\frac{\pi}{\beta}+\frac{d}{2}-1\right) 2^{\frac{\pi}{\beta}+\frac{d}{2}-1} x^{-\frac{2\pi}{\beta}+1} -x^{d-1} R\right) \nonumber\\
		&\gtrsim x^{-\frac{2\pi}{\beta}+1}, \quad x\ra\infty \label{asympforp0m}
		\end{align}
		with
		\begin{align*}
		R:=   \int_0^\infty e^{-y\frac{x^2}{2}}  \frac{y^{2\pi/\beta+d/2-2}}{y^{\pi/\beta}+1} dy	&\leq   \int_0^\infty e^{-y\frac{x^2}{2}} y^{\frac{\pi}{\beta}+\frac{d}{2}-\frac{3}{2}} dy\lesssim x^{-\frac{2\pi}{\beta}-d+1},\quad x\ra\infty.
		\end{align*}
		For the next step let $a\in\{0,(2\pi/\beta)-1\}$. We apply Theorem \ref{parseval} and the rule $\M[(\cdot)^a f(\cdot)](z)=\M[f(\cdot)](z+a)$ and obtain
		\begin{align}
		\int_0^\infty x^{a} (p_{0,M}&(x)-p_{1,M}(x))^2 dx \nonumber\\
		=&~ \frac{\delta^2 2^\frac{3+a}{2}}{\pi i\Gamma\left({d}/{2}\right)^2} \int_{\gamma-i\infty}^{\gamma+i\infty} \Gamma\left(\frac{z+d-1}{2}\right) \M[q\odot \rho_M]\left(\frac{z+1}{2}\right) \label{instep2} \\ 
		&\times \Gamma\left(\frac{a+d-z}{2}\right) \M[q\odot \rho_M]\left(\frac{2+a-z}{2}\right) dz \nonumber
		\end{align}
		for suitable $\gamma$, where $\M[q\odot \rho_M]=\M[q] \M[\rho_M]$. Due to \eqref{mellinvonrhoM}, we can estimate
		\begin{equation}\label{abschmrho}
		|\M[\rho_M](u+iv)|  \leq e^{\frac{(u-1)^2}{2}}\frac{\varphi(v+M) +\varphi(v-M)}{2}
		\end{equation}
		with $\varphi(v)=e^{-\frac{v^2}{2}}$. Next, we use Lemma \ref{abschgamma0} in \eqref{instep2} to estimate the gamma terms, then plug in \eqref{abschmrho} and $|\M[q](u+iv)| \leq ce^{-\beta|v|}$ for some $c>0$ to obtain
		\begin{equation}\label{step2}
		\int_0^\infty x^{a} (p_{1,M}(x)-p_{0,M}(x))^2 dx 
		\lesssim M^\frac{2d+a-3}{2} e^{-M\pi(1+2\beta/\pi)} , \quad M\ra\infty
		\end{equation}
		for $a\in\{0,(2\pi/\beta)-1\}$. By \eqref{asympforp0m} and \eqref{step2},
		\begin{align*}
		\chi^2(p_{1,M}|p_{0,M}) &= \int_0^\infty \frac{(p_{M,1}(x)-p_{M,0}(x))^2}{p_{M,0}(x)}dx \\
		&\lesssim \int_{0}^{\infty} (p_{M,1}(x)-p_{M,0}(x))^2 dx 
		+ \int_{0}^\infty x^{-\frac{2\pi}{\beta}+1} (p_{M,1}(x)-p_{M,0}(x))^2 dx\\
		&\lesssim M^{(2d-3)/2} e^{-M(\pi+2\beta)} +M^{(\pi/\beta)+d-2}  e^{-M(\pi+2\beta)}, \quad M\ra\infty,
		\end{align*}
		where $M^{(\pi/\beta)+d-2} e^{-M(\pi+2\beta)}$ is the dominating term. This proves the lemma.
	\end{proof}
	Lemma \ref{chi2distance1} implies \eqref{modifiedchisquredcrit}. With the choice
	$$ M= \frac{\log(n)}{\pi+2\beta}$$
	Lemma \ref{supdistf0f1} implies \eqref{dxkrit}. Claim of Theorem \ref{optimalitysatz}(i) follows with Theorem \ref{reductiontotwohypos}.

\section{Appendix}
For proof of Lemmas \ref{gammaarggegeninf} and \ref{abschgamma0} we refer to \cite{andrews}.
\begin{lem}\label{gammaarggegeninf}
	For $|\arg(s)|\leq \pi$ we have $ \Gamma(s) \sim \sqrt{2\pi}s^{s-1/2}e^{-s}$ for $|s|\ra\infty.$
\end{lem}

\begin{lem} \label{abschgamma0}
	For all $\alpha\in\R$ there are $C_1,C_2\geq 0$ such that 
	$$ C_1|\beta|^{\alpha-1/2}e^{-|\beta|\pi/2} \leq |\Gamma(\alpha+i\beta)| \leq C_2 |\beta|^{\alpha-1/2}e^{-|\beta|\pi/2}, \quad |\beta|\geq 2.$$
	\end{lem}

\begin{kor} \label{abschgamma}
	\begin{itemize}
		\item[(i)] For all $\alpha \in(0,1/2), \delta>0$ and $U>2$ there is a $C(\alpha,\delta)>0$ such that \[\int_{-U}^U \frac{1}{|\Gamma(\alpha+i v)|^\delta}dv \leq C(\alpha,\delta) U^{(1/2-\alpha)\delta}e^{U\pi\delta/2}. \]
		\item[(ii)] For all $\alpha\geq 1/2,~\delta>0$ and $U>2$ there are $C_1(\alpha,\delta)$ and $C_2(\alpha,\delta)>0$ with
		\[\int_{-U}^U \frac{1}{|\Gamma(\alpha+i v)|^\delta}dv \leq C_1(\alpha,\delta) + C_2(\alpha,\delta) e^{U\pi\delta/2}. \]
	\end{itemize}
\end{kor}
\begin{proof}
Define $C:=\int_{-2}^{2} \frac{1}{|\Gamma(\alpha+i v)|^\delta}dv $. For $\alpha \in(0,1/2)$ Lemma \ref{abschgamma0} gives a $C_1 >0$ such that
	\begin{align*}
	\int_{-U}^U \frac{1}{|\Gamma(\alpha+i v)|^\delta}dv
	&\leq  C +  C_1 \int_{\{2<|v|<U\}} |v|^{(1/2-\alpha)\delta}e^{|v|\pi\delta/2} dv\\
	&\leq  C + 2C_1 U^{(1/2-\alpha)\delta} \int_2^U e^{v\pi\delta/2} dv\\
	&= C +  4C_1 (\pi\delta)^{-1} U^{(1/2-\alpha)\delta} ( e^{U\pi\delta/2} - e^{\pi\delta})
	\end{align*}
	which implies the claim with $C(\alpha,\delta):=\max\{2C, 8C_1 (\pi\delta)^{-1}\}$. The case $\alpha\geq\frac{1}{2}$ follows similarly with $C_1(\alpha,\delta):=C$ and $C_2(\alpha,\delta):= 4 C_1 (\pi\delta)^{-1}$.
\end{proof}

\begin{lem}\label{erdelyi}
	Let $\alpha<\beta$. If $f:(\alpha,\beta)\ra\C$ is $N$ times continuously differentiable ($N\in\N$), then we have the expansion
	$$ \int_{\alpha}^\beta f(u)e^{ixu}du= \sum\limits_{n=0}^{N-1} \frac{i^{n-1}}{x^{n+1}}f^{(n)}(\beta)e^{ix\beta}-\sum\limits_{n=0}^{N-1} \frac{i^{n-1}}{x^{n+1}}f^{(n)}(\alpha)e^{ix\alpha}+o(x^{-N}),\quad x\ra\infty.$$
	\end{lem}
\begin{proof}
See \cite[page 47]{erdelyi}.
\end{proof}

\section*{Acknowledgement}
The author was supported by the Deutsche Forschungsgemeinschaft (DFG) via RTG 2131 \textit{High-dimensional Phenomena in Probability – Fluctuations and Discontinuity}.

\bibliographystyle{spbasic}
\bibliography{wa}

\end{document}